\documentclass[12pt,  reqno]{amsart}
\usepackage{amsmath, amsthm, amscd, amsfonts, amssymb, graphicx, color}

\newtheorem{theorem}{Theorem}[section]
\newtheorem{lemma}[theorem]{Lemma}
\newtheorem{proposition}[theorem]{Proposition}

\theoremstyle{definition}
\newtheorem{definition}[theorem]{Definition}
\newtheorem{example}[theorem]{Example}

\theoremstyle{remark}
\newtheorem{remark}[theorem]{Remark}
\numberwithin{equation}{section}

\begin{document}
\setcounter{page}{1}

\title[numerical radius of a quaternionic normal operator ]{on the numerical radius of a quaternionic normal operator}

\author{G. Ramesh$^{*}$}

\address{Department of Mathematics\\I. I. T. Hyderabad, Kandi(V) \\  Sangareddy, Telangana \\ India-502 285.}

\email{rameshg@iith.ac.in}

\thanks{}

\date{\today}

\date{\today
\newline \indent $^{*}$ Corresponding author}
\begin{abstract}
We prove that  for a  right linear  bounded normal operator on a quaternionic Hilbert space (quaternionic bounded normal operator) the norm and the numerical radius are equal. As a consequence of this result we  give a new proof of the known fact that a non zero  quaternionic compact normal operator has a non zero right eigenvalue.  Using this we give
a new proof of the spectral theorem for quaternionic compact normal
operators.. Finally, we show that every quaternionic compact operator  is norm attaining and prove the Lindenstrauss theorem on norm attaining operators, namely, the set of all norm attaining quaternionic operators is norm dense in the space of all bounded quaternionic operators defined between two quaternionic Hilbert spaces.
\end{abstract}
\maketitle
\section{Introduction}
     It is well known that  for a bounded normal operator on a complex  Hilbert space,  the norm and the numerical radius are  the same. In this note, we prove this result for right linear normal operators on a quaternionic Hilbert space. As a consequence of this result, we show that every compact normal operator on a quaternionic Hilbert space has a non zero right eigenvalue. This is a crucial point in proving the spectral representation theorem for such operators.

The spectral theorem for compact normal operators on a  quaternionic Hilbert space is  appeared in a recent article by Ghiloni etal \cite{Ghiloni1}. The authors mainly used the left multiplication on the space of all bounded right linear operators on a  quaternionic Hilbert space.  The spectral theorem for  matrices with quaternionic entries was studied in \cite{Farenick}. In this article, we give a new proof of
the spectral theorem for compact operators on general quaternionic
Hilbert spaces (see \cite[Theorem 1.2]{Ghiloni1}).  First, we prove that the norm and the numerical radius of a quaternionic normal operator are the same. To prove this, we associate a unique complex normal operator with the given quaternionic normal operator which preserve the norm and the numerical radius. Using this technique and the  classical result, we obtain the result.  As a consequence, we prove that   a quaternionic compact normal operator has a non zero right eigenvalue. Finally, with this idea, we give a new proof of the
spectral theorem for quaternionic compact normal operators. Later, we extend the Lindenstrauss theorem on norm attaining operators to the quaternionic case. A simple proof in the classical case can
be found in \cite{enfloetal}.


Organization of the article: In the second section we give necessary details of quaternionic Hilbert spaces and right linear operators on such spaces. In the third section we prove that for a normal quaternionic operator the norm and the numerical radius are equal. Using this we prove the spectral theorem for quaternionic compact operators. In the final section, we consider the norm attaining operators and prove the well known Lindenstrauss theorem on  norm attaining operators in the case of quaternionic operators.

\section{Preliminaries}

We denote the division ring of real quaternions by $\mathbb H$. If $q\in \mathbb H$, then $q = q_{0} + q_{1}i+q_{2}j+q_{3}k$, where $q_{n} \in \mathbb{R}$ for $n = 0, 1, 2, 3$ and $i,j,k$ satisfy the following conditions:
\begin{equation*}			
 i^{2}=j^{2}=k^{2}=-1,\; ij = -ji = k,\;  jk = -kj = i,\; \text{and}\;  ki = -ik = j.
\end{equation*}
 The conjugate of $q$ is $ \overline{q} = q_{0}-q_{1}i-q_{2}j-q_{3}k$ and $ |q| := \sqrt{q_{0}^{2}+q_{1}^{2}+q_{2}^{2}+q_{3}^{2}}$. The imaginary part of $\mathbb{H}$ is defined by  $Im(\mathbb{H}) = \left\{ q \in \mathbb{H} : q = -\overline{q}\right\}.$ The set of all unit imaginary quaternions is denoted by $\mathbb{S}$, that is $ \mathbb{S}:= \left\{ q \in Im(\mathbb{H}): |q| = 1 \right\}$ and  the unit sphere of $\mathbb H$ by $S_{\mathbb H}$.

Here we list out some of the properties of quaternions, which we need later.

\begin{enumerate}
 \item For $p,q \in \mathbb{H}$, we have $ \overline{pq} = \overline{q} \overline{p},\; |pq| = |p||q|$ and $|\overline{p}| = |\overline{q}|$.
\item We define an equivalence relation on $\mathbb{H}$ as, $p \sim q$ if and only if $ p = s^{-1}qs $ for some $ s \neq 0 \in \mathbb{H} $. The equivalence class of $p$ is $  [p]:= \{s^{-1}qs : 0 \neq s \in H\}$.
\item For each $m \in \mathbb{S}, \mathbb{C}_{m}: =\left\{ \alpha + m \beta : \alpha, \beta \in \mathbb{R} \right\} $ is a real subalgebra of $\mathbb{H}$ and is called as the slice complex plane generated by $1$ and $m$.
\item We have  $\mathbb{C}_{m} \cap \mathbb{C}_{n} = \mathbb{R}$ if $ m \neq \pm{n}$, and  $\mathbb{H} = \displaystyle \cup_{m \in \mathbb{S}} \mathbb{C}_{m}$.
\end{enumerate}

A right $\mathbb H$-module $H$ is called a quaternionic pre-Hilbert space if there exists
a Hermitian quaternionic scalar product; namely a function $\langle\cdot\rangle: H\times H\rightarrow \mathbb H$
satisfying the following:
\begin{enumerate}
\item  $\langle u,vp + wq\rangle = \langle u,v\rangle \, p + \langle u,w\rangle \,q$  for all  $u, v,w \in H$ and $p, q \in \mathbb H$
\item  $\langle u,v\rangle = \overline{\langle v,u\rangle}$  for all $u, v \in H$
\item $\langle u,u\rangle\geq 0 $ for all $u\in H$ and $\langle u,u\rangle=0$ iff $u=0$.
\end{enumerate}

Let $ H $ be a quaternionic pre-Hilbert space with Hermitian quaternionic scalar product  $ \left\langle \cdot   \right\rangle $ on $ H$.  Such an inner product $\langle \cdot \rangle$ satisfies the
Cauchy-Schwarz inequality:
\begin{equation*}
|\langle u,v\rangle|^2\leq \langle u,u\rangle \, \langle v,v\rangle\;  \text{for all}\; u,v\in H.
\end{equation*}
 Define $ \| u \|
= {\left\langle u, u \right\rangle}^{\frac{1}{2}}, $ for every $ u \in H.$  Then $\|\cdot\|$ is a norm in the usual real sense.  If the normed space $ (H, \|\cdot\|) $ is  complete, then $H$ is called a quaternionic Hilbert space.

The norm induced by this inner product satisfy the parallelogram law:
\begin{equation*}
\|u+v\|^2=\|u-v\|^2=2(\|u\|^2+\|v\|^2)\; \text{for all}\; u,v\in H.
\end{equation*}

We denote the unit sphere of the Hilbert space $H$ by $S_H$.

An operator $T:H\rightarrow H$ is said to be right linear if

\begin{enumerate}
\item $T(x+y)=Tx+Ty$  for all $x,y\in H$

\item $T(xq)=(Tx)q$ for all $x\in H,\; q\in \mathbb H$.
\end{enumerate}
A right linear operator $T:H\rightarrow H$ is said to be bounded if there exists a $M>0$ such that
$\|Tx\| \leq M\, \|x\|$ for all $x\in H$. For such an operator the norm is defined by
\begin{equation*}
\|T\| = \sup \left\{ \|Tu\| : u \in S_H \right\}.
\end{equation*}
We denote the space of all bounded right linear operators on $H$ by $\mathcal B(H)$. For $T\in \mathcal B(H)$, the null space is defined by $N(T)={\{x\in H:Tx=0}\}$ and the range space is defined by
$R(T)={\{Tx:x\in H}\}$.

Let $ T \in \mathcal{B}(H). $ Then there exists a unique operator $T^{*} \in \mathcal{B}(H)$ such that $ \left\langle u,  Tv\right\rangle = \left\langle T^{*}u ,v\right\rangle $ for all $ u,v \in H. $ This operator $ T^{*} $ called the adjoint of $T.$

Let $T \in \mathcal{B}(H).$ Then $T$ is said to be self-adjoint if $ T = T^{*}$, an  anti self-adjoint if $T^{*}= -T$, normal if $ TT^{*} = T^{*}T $ and  unitary if $ TT^{*} = T^{*}T = I$. If $T$ is self-adjoint and $\langle x,Tx\rangle \geq 0$ for all $x\in H$, then $T$ is said to be positive.

 Let $T \in \mathcal{B}(H)$ be positive. Then there exists  a unique positive operator $S \in \mathcal{B}(H)$ such that $S^{2}=T.$ Such a $S$ is called the square root of $T$ and is denoted by $S = T^{\frac{1}{2}}.$

 If $S\in \mathcal B(H)$. Then the operator $|S|:=(S^*S)^{\frac{1}{2}}$ is called as  the modulus of $S$. In fact, there exists a partial isometry  $V$  ($\|Vx\|=\|x\|$ for all $x\in N(V)^{\bot}$) such that $T=V|T|$ and $N(V)=N(T)$. This decomposition is unique and is known as the \textit{polar decomposition} of $T$ (We refer to \cite[Theorem 2.20]{Ghiloni} for more details).

Let  $ T \in \mathcal{B}(H)$ and $ q \in \mathbb{H}.$ Define
\begin{equation*}
 \Delta_{q}(T):= T^{2}-T(q+\overline{q})+I|q|^{2},
\end{equation*}
where, if $r\in \mathbb R$, the operator  $Tr\in \mathcal B(H)$ is defined by setting $(Tr)x:=(Tx)r$ for all $x \in H$.
The spherical spectrum of $T$ is defined as
\begin{equation*}
\sigma_{S}(T):= \{ q \in \mathbb{H}: \Delta_{q}(T) \; \text{ is not invertible in} \; \mathcal{B}(H)\}.
\end{equation*}
The spherical point spectrum is defined as
\begin{equation*}
\sigma_{pS}(T) := {\{ q \in \mathbb{H}: \Delta_{q}(T)\; \text{is not one-to-one }}\}.
\end{equation*}
All the above mentioned material can be found in \cite{Ghiloni}.

The numerical range and the numerical radius of $T$ is defined by
\begin{align*}
W(T)&={\{\langle x,Tx\rangle:x\in S_H}\},\\
w(T)&=\sup{\{|\langle x,Tx\rangle|:x\in S_H}\},
\end{align*}
respectively.

For a normal operator on a complex Hilbert space the numerical range is convex, whereas, this is not the case for normal operators on quaternionic Hilbert spaces (see \cite{yeung} for details).

Let $T \in \mathcal{B}(H).$ Then $T$  is said to be compact if ${T(B)}$ is pre-compact for every bounded subset $B$ of $H.$ Equivalently, $\{T(x_{n})\}$ has a convergent subsequence for every bounded sequence $\{x_{n}\} $ of $ H$.
\section{Numerical radius of a normal operator}

Suppose that $H$ is a non zero quaternionic Hilbert space with Hermitian quaternionic scalar product $\langle \cdot, \cdot \rangle$. Let  $m \in \mathbb{S}$ and $ J \in \mathcal{B}(H)$ be an anti self-adjoint, unitary operator. Define $H_{\pm}^{Jm}:={\{u\in H:Ju=\pm um}\}$. Then $H_{\pm}^{Jm}$ is a non-zero closed subset of $H$. The restriction of the inner product on $H$ to $H_{\pm}^{Jm}$ is a $\mathbb C_m$-valued inner product and with respect to this inner product $H_{\pm}^{Jm}$ is a Hilbert space. In fact, if we consider $H$ as a $\mathbb C_m$ linear space, $H$ has the decomposition: $H=H_{+}^{Jm}\oplus H_{-}^{Jm}$ (see \cite[pages 21-22]{Ghiloni} for details).
We need the following results to prove our main theorem.
\begin{proposition} \cite[Proposition 3.11]{Ghiloni}\label{extension}
If $T \colon H^{Jm}_{+} \to H^{Jm}_{+} $ is a bounded $\mathbb{C}_{m}-$ linear operator, then there exists unique bounded, right $\mathbb{H}-$ linear operator $\widetilde{T}\colon H \to H$ such that $\widetilde{T}(u) = T(u),$ for every $u \in H^{Jm}_{+}.$

Furthermore,

\begin{enumerate}
\item $\|\widetilde{T}\| = \|T\|$
\item $J\widetilde{T} = \widetilde{T} J$
\item $(\widetilde{T})^{*} = \widetilde{T^{*}}$
\item If $S \colon H^{Jm}_{+} \to H^{Jm}_{+}$ is a bounded $\mathbb{C}_{m}-$ linear operator, then $\widetilde{ST} = \widetilde{S} \widetilde{T}$
\item If $S$ is the inverse of $T,$ then $\widetilde{S}$ is the inverse of $\widetilde{T}.$

\end{enumerate}
On the other hand,   if $V\in \mathcal B(H)$, then  there exists a unique $U\in \mathcal B(H_{+}^{Jm})$ such that $\widetilde{U}=V$ if and only if $JV = VJ$.
\end{proposition}

If $T$ is normal(but not self-adjoint), there exists an anti self-adjoint, unitary $J\in \mathcal B(H)$ such that $TJ=JT$ (see \cite[Theorem 5.9]{Ghiloni} for details). Hence Proposition \ref{extension} holds with $V=T$.
If $T$ is self-adjoint, then the existence of an anti self-adjoint, unitary $J\in \mathcal B(H)$ such that $TJ=JT$ is guaranteed by \cite[Theorem 5.7(b)]{Ghiloni}.

\begin{remark}\label{sumextension}
If $S,T \in \mathcal B( H^{Jm}_{+})$, then it can be easily shown that $\widetilde{S+T}=\tilde S+\tilde T$ by following the same steps as in  \cite[Proposition 3.11]{Ghiloni}.
\end{remark}

\begin{theorem}\label{normaloid}
Let $T\in \mathcal B(H)$ be normal. Then $w(T)=\|T\|$.
\end{theorem}
\begin{proof}
First note that $S_{H^{Jm}_{+}}\subseteq S_H$.  Let $T_{+}\in \mathcal B(H^{Jm}_{+} )$  be such that  $\widetilde {T_{+}}=T$ as in Proposition \ref{extension}.
Then
\begin{align*}
w(T)=\sup{\{|\langle Tx,x\rangle |:x\in S_H}\}&\geq \sup{\{|\langle Tx,x\rangle |:x\in S_{H^{Jm}_{+}} }\}\\
                                                                        &= \sup{\{|\langle T_{+}x,x\rangle |:x\in S_{H^{Jm}_{+}} }\}\\
                                                                        &=w(T_{+}).
\end{align*}
Since $T_{+}$ is normal, we have $\|T_{+}\|=w(T_{+})$. But $\|T_{+}\|=\|T\|$. This shows that $w(T)\geq \|T\|$. But the other inequality is clear. Thus $w(T)=\|T\|$.
\end{proof}

As a consequence we obtain a new proof of the following known result.
\begin{theorem}\label{existenceofeigenvalue} \cite[Theorem 1.1]{Ghiloni1}
If $T\in \mathcal B(H)$ is compact and normal, then  there exists a $q\in \sigma_{pS}(T)$ such that $|q|=\|T\|$.
\end{theorem}
\begin{proof}
If $T=0$, then it suffices to set $q=0$. Suppose $T\neq 0$. By Theorem \ref{normaloid}, there exists a sequence $(x_n)$ in $S_H$ such that $|\langle x_n,Tx_n\rangle |\rightarrow  \|T\|$ as $n\rightarrow \infty$.  If necessary, choose a subsequence of $(x_n)$, ( we again denote it by $(x_n)$) such that $\langle x_n,Tx_n\rangle  \rightarrow q$ for some $q\in \mathbb H\setminus {\{0}\}$ with $|q|=\|T\|$. Since $T$ is compact, there exists a subsequence $(x_{n_k})$ of $(x_n)$ such that $Tx_{n_k}$ is convergent. Let $y:=\displaystyle \lim_{k\rightarrow \infty}Tx_{n_k}$. Observe that $\|Tx_k\|\leq \|T\|\|x_k\|\leq \|T\|$ for every $k$. It follows that $\|y\|\leq \|T\|=|q|$. Then
\begin{align*}
\|Tx_{n_k}-x_{n_k}q\|^2&=\langle Tx_{n_k}-x_{n_k}q,Tx_{n_k}-x_{n_k}q\rangle \\
                       &=\|Tx_{n_k}\|^2-\overline{\langle x_{n_k},Tx_{n_k}}\rangle q-\bar{q}\langle x_{n_k},Tx_{n_k}\rangle+|q|^2\\
                       &\rightarrow \|y\|^2-|q|^2\leq 0.
\end{align*}
Hence $\|y\|=|q|$; in particular, $y\neq 0$.  So $T(x_{n_k})-x_{n_k}q\rightarrow 0$ as $n\rightarrow \infty$. Since $(Tx_{n_k})$ converges to $y$, it follows that $x_{n_k}q\rightarrow y$.  Thus $Ty=\lim_{k\rightarrow \infty}T(x_{n_k}q)=\lim_{k\rightarrow \infty}T(x_{n_k})q=yq$. Thanks to Proposition $4.5$ of \cite{Ghiloni1}, we have that $q\in \sigma_{pS}(T)$.
\end{proof}

Using Theorem \ref{existenceofeigenvalue}  as in
the case of complex compact operators, we can give a new proof of the
spectral theorem for quaternionic compact normal operators (see \cite[Theorem 1.2] {Ghiloni1}).

\begin{definition}
Let $H_0$ be a quaternionic closed subspace of a quaternionic Hilbert space $H$. Then $H_0$ is said to be invariant under $T\in \mathcal B(H)$ if $T(H_0)\subseteq H_0$. If $H_0$ and $H_0^{\bot}$ are both invariant under $T$, then $H_0$ is said to be a reducing subspace for $T$.
\end{definition}
\begin{example}\label{onedmiminvsubsp}
Let $T\in \mathcal B(H)$ and $T\phi=\phi q$, where $|q|=\|T\|$ and $\phi\in S_H$. Then $H_0:=\displaystyle \text{span}_{\mathbb H}{\{\phi}\}$ is a non trivial reducing subspace for $T$. As $H_0$
is right linear we can see that $H_0$ is an invariant subspace for $T$. To show that $H_0$ reduces $T$,   it is enough to prove  $H_0$ to be invariant under $T^*$. Note that
\begin{align*}
|q|^2=\|T\phi\|^2=\langle T^*T\phi,\phi\rangle=\bar{q}\langle T^*\phi, \phi\rangle.
\end{align*}
Thus $\langle T^*\phi,\phi\rangle=q$. As $|\langle T^*\phi,\phi\rangle|=|q|$, we have $T^*\phi=\phi p$ for some $p\in \mathbb H$.  Then $q=\langle T^*\phi,\phi\rangle =\langle \phi p,\phi\rangle=\bar{p}$. Thus $p=\bar{q}$. As $T^*$ is right linear, we can conclude that $T^*(H_0)\subseteq H_0$.
\end{example}
\begin{lemma}\label{propertiesrestrictedop}
Let $T\in \mathcal B(H)$ be normal and $H_0$, a reducing subspace for $T$. Let $T_0:=T|_{H_0}$. Then
\begin{enumerate}
\item \label{restrictedadjoint}$T_0^*=T^*|_{H_0}$
\item \label{restrictednormality} $T_0$ is normal.
\end{enumerate}
\end{lemma}
\begin{proof}
Let $x,y\in H_0$. Then
\begin{equation*}
 \langle T_0^*x,y\rangle= \langle x,T_0y\rangle =\langle x,Ty\rangle = \langle T^*x,y\rangle.
\end{equation*}
We can conclude that $T_0^*x-T^*x\in H_0^{\bot}$. Also, since $H_0$ reduces both $T$ and $T^*$, it follows that
$T_0^*x-T^*x\in H_0$. That is $T_0^*x=T^*x$ for each $x\in H_0$. This completes the proof of (\ref{restrictedadjoint}).

To prove (\ref{restrictednormality}), let $x\in H_0$. Then we have
\begin{align*}
 \langle T_0^*T_0x,x\rangle=\langle Tx,Tx\rangle& =\langle x,T^*Tx\rangle\\
                                                                            & =\langle x,TT^*x\rangle\\
                                                                            &=\langle TT^*x,x\rangle\\
                                                                            &=\langle T_0T^*|_{H_0}x,x\rangle\\
                                                                            &=\langle T_0T_0^*x,x\rangle.
\end{align*}
Now the conclusion follows from the polarization identity.
\end{proof}

\begin{theorem}\label{spectralthm}
Let $T \in \mathcal B(H)$ be compact and normal. Then there exists a system of eigenvectors $\left\{\phi_{n} \right\}$ and corresponding right quaternion eigenvalues $\left\{q_{n}\right\}$  such that

\begin{equation}\label{spectralrepneqn}
  Tu = \sum_{n=1}^{\infty} \phi_{n}q_{n} \left\langle \phi_{n} | u \right\rangle,  \text{for all}\;   u \in H.
\end{equation}
Moreover, if $\{q_{n}\}$ is infinite, then $ q_{n} \rightarrow 0$ as $n \rightarrow \infty$.

The series on the right hand side of Equation (\ref{spectralrepneqn}) converges in the operator norm of $\mathcal B(H)$.
\end{theorem}

\begin{proof}
If $T=0$, then there is nothing to prove. Hence assume that $T\neq 0$. Set $T_1=T$ and $H_1=H$.
Since $T_1$ is compact and normal by Theorem \ref{existenceofeigenvalue}, there exists a $\phi_1\in H_1 \setminus {\{0}\}$ and $q_1\in \mathbb H\setminus {\{0}\}$ such that $T\phi_1=\phi_1q_1$. Also, note that $\|T_1\|=|q_1|$. By  Example \ref{onedmiminvsubsp}, the space $H_1:=\displaystyle \text{span}_{\mathbb H}{\{\phi_1}\}^{\bot}$ is a reducing subspace for $T$. Next, let $T_2:=T_1|_{H_2}$. Then either $T_2=0$ or $T_2\neq 0$. In case if $T_2=0$, there is nothing to proceed further. If $T_2\neq 0$, then $T_2$ is normal by (\ref{restrictednormality}) of Lemma \ref{propertiesrestrictedop}.
Thus, again by Theorem \ref{existenceofeigenvalue} there exists $q_2\in \mathbb H\setminus {\{0}\}$ with $|q_2|=\|T_2\|,\; \phi_2\in H_2\setminus {\{0}\}$ such that $T\phi_2=T_2\phi_2=\phi_2q_2$.
Note that $|q_2|\leq |q_1|$ and $\phi_1$ and $\phi_2$ are orthogonal by construction.

Let $H_3:=\displaystyle \text{span}_{\mathbb H}{\{\phi_1,\phi_2}\}^{\bot}$. Since $T$ is normal and $H_3$ is a reducing subspace for $T_3:=T|_{H_3}$, we have that $T_3$ is normal and compact. Now either $T_3=0$ or $T_3\neq 0$. If $T_3\neq 0$, then there exists a quaternion  $q_3\in \mathbb H\setminus {\{0}\}$ and $\phi_3\in H_3\setminus {\{0}\}$ such that $T\phi_3=T_3\phi_3=\phi_3q_3$ and $|q_3|\leq |q_2|$. Bye construction we have that $\phi_3$ is orthogonal to both $\phi_1$ and $\phi_2$.

Proceeding in this way, we  end up with either $T_n=0$ for some $n\in \mathbb N$ or there exists a sequence $(q_n)$ of non zero quaternions and a sequence of vectors $(\phi_n)\subset H$ satisfying:
\begin{enumerate}
\item $T\phi_n=\phi_nq_n$ and $|q_n|=\|T_n\|$ for each $n\in \mathbb N$,
\item $|q_{n+1}|\leq |q_n|$ for each $n\in \mathbb N$,
\item $\phi_r$ is orthogonal to $\phi_s$ for each $r,s\in \mathbb N $ and $r\neq s$.
\end{enumerate}
Next, we claim that  if $(q_n)$ is infinite, then $q_n\rightarrow 0$ as $n\rightarrow \infty$.
If this is not the case, there exists $\epsilon>0$ such that $|q_n|>\epsilon$ for infinitely many $n\in \mathbb N$. Let $S={\{r\in \mathbb N: |q_r|>\epsilon}\}$. Then we have $T\phi_r=\phi_rq_r$ for each $r\in S$. Thus, for $r,s\in S$, we have
\begin{align*}
\|T\phi_r-T\phi_s\|^2&=\|\phi_rq_r-\phi_sq_s\|^2\\
                    &=|q_r|^2+|q_s|^2\\
                    &>2\epsilon^2.
\end{align*}
This shows that $(T\phi_r)$ is not Cauchy in $H$. But, this contradicts the fact that $T$ is compact. Hence our assumption that $q_n\nrightarrow 0$ is wrong.

Next, we obtain the representation of $T$ as in Equation (\ref{spectralrepneqn}).

For $x\in H$, define $x_n:=x-\displaystyle \sum_{r=1}^{n-1}\phi_r\langle \phi_r,x\rangle$ for each $n\in \mathbb N$. Then $\langle x_n,\phi_r\rangle =0$ for each $r=1,2,\dots,n-1$.

We have the following two cases:\\

Case $1$: $T_n=0$ for some $n\in \mathbb N$

In this case, we have $0=T_nx_n=Tx_n$. Thus, $Tx=\displaystyle  \sum_{r=1}^{n-1}\phi_rq_r\langle \phi_r,x\rangle$ for each $x\in H$. That is, $T$ is a finite rank operator with rank $n-1$.

 Case $2$: $T_n\neq 0$ for any $n\in \mathbb N$\\
Since $x_n\in H_{n}^{\bot}$ for each $n\in \mathbb N$, it can be easily checked by the Pythagorean property that
$\|x_n\|\leq \|x\|$ for each $n\in \mathbb N$. Thus,
\begin{align*}
 \|Tx-\displaystyle \sum_{r=1}^n \phi_r q_r\langle \phi_r,x\rangle\|=\|T_nx_n\|&\leq |q_n|\|x_n\|\\
                                                                    &\leq |q_n|\|x\|\\
                                                                    &\rightarrow 0 \;\text{as}\; n\rightarrow \infty.
\end{align*}
That is $Tx=\displaystyle \sum_{n=1}^{\infty} \phi_n q_n\langle \phi_n,x\rangle$  for each $x\in H$.
\end{proof}
\begin{remark}
Note that if $q$ is a right eigenvalue for $T$ and $p\in [q]$, then $p$ is also a right eigenvalue for $T$. Hence by Theorem \ref{spectralthm}, we have that $\sigma_{pS}(T)={\{[q_n]:n\in \mathbb N}\}$ and $\sigma_{S}(T)\subseteq \sigma_{pS}(T)\cup {\{0}\}$.
\end{remark}

\section{Norm attaining Operators}

In this section we extend the Lindenstrauss theorem on norm attaining operators from the classical case to the quaternionic case. Explicitly, we show that the set of all quaternionic norm attaining   operators is dense in the
space of all bounded quaternionic operators with respect to the operator norm.

Recall that a bounded right linear operator $T$ is said to be norm attaining if there exists a $x_0\in S_H$ such that $\|Tx_0\|=\|T\|$. As $\|Tx\|=\||T|x\|$ for all $x\in H$ and $\|T\|=\||T|\|$, it follows that $T$ is norm attaining if and only if $|T|$ is norm attaining.

     We show that every  quaternion compact operator is norm attaining. In the case of operators on complex Hilbert space, this  can be proved by the help of  Banach-Alaouglu's theorem. In our  case we prove it by using Theorem \ref{existenceofeigenvalue}.

We denote the set of all norm attaining operators defined between $H_1$ and $H_2$ by $\mathcal N(H_1,H_2)$ and $\mathcal N(H,H)$ by $\mathcal N(H)$.
\begin{proposition}
Let $T\in \mathcal B(H)$ be compact. Then $T\in \mathcal N(H)$.
\end{proposition}
\begin{proof}
Since $T$ is compact, $|T|$ is compact as well (see \cite{Fashandi}). The operator $|T|$ is also self-adjoint and hence normal because it is positive by definition  (see \cite[Proposition 2.17(b)]{Ghiloni1}). Hence by Theorem \ref{existenceofeigenvalue}, $\|T\|$ is a right eigenvalue. Hence   the result follows.
\end{proof}
\begin{lemma}\label{normattainingequivalnce}
Let $T\in \mathcal B(H)$ be normal  and $T_{+}$ be such that $\tilde T_{+}=T$ as in Proposition \ref{extension}. Then  $T_{+}\in \mathcal N(H^{Jm}_{+})$   if and only if  $T\in \mathcal N(H)$.
\end{lemma}

\begin{proof}
If $T_{+}\in \mathcal N(H^{Jm}_{+})$, then there exists $x_0\in S_{H^{Jm}_{+}}$  such that $\|T_{+}x_0\|=\|T_{+}\|$. Since $\|T_{+}\|=\|T\|$, the conclusion follows.

On the other hand, suppose $T\in \mathcal N(H)$. Choose $x_0\in S_{H}$ such that $\|Tx_0\|=\|T\|$. Let $y_0:=\dfrac{1}{\sqrt{2}}(x_0-(Jx_0)m)$. Then $y_0\in H^{Jm}_{+}$. Also,
\begin{align*}
\|y_0\|^2&=\dfrac{1}{2}\|(x_0-(Jx_0)m)\|^2\\
                &=\dfrac{1}{2}\langle (x_0-(Jx_0)m),(x_0-(Jx_0)m)\rangle\\
                &=\dfrac{1}{2} \Big(\langle x_0,x_0\rangle -\langle (Jx_0)m,x_0\rangle-\langle x_0,(Jx_0)m\rangle+\langle (Jx_0)m,(Jx_0)m\rangle\Big)\\
                &=\dfrac{1}{2} \Big(1 -\langle (Jx_0)m,x_0\rangle-\langle x_0,(Jx_0)m\rangle+1\Big).
                \end{align*}
Note that as $J$ is anti self-adjoint,  $\langle x_0,Jx_0\rangle=-  \overline{\langle x_0,Jx_0\rangle}$. With this, we have $\langle (Jx_0)m,x_0\rangle=-\langle x_0,Jx_0\rangle m$. Hence $\|y_0\|=1$.

Next using the fact that $J$ commutes with $T$ and $T^*$ and $T(H^{Jm}_{+})\subseteq H^{Jm}_{+}$, we can conclude that
\begin{equation*}
\|T_{+}y_0\|^2=\|Tx_0\|^2=\|T\|^2=\|T_{+}\|^2.
\end{equation*}
Thus $T_{+}$ attains norm at $y_0$.
\end{proof}
\begin{proposition}\label{positivenormattaining}
Let $B\in \mathcal B(H)$ be positive. Then  for given $\epsilon>0$, there exists a rank one positive operator $C$ with  $\|C\|\leq \epsilon$ and $y\in N(B)^{\bot}$ such that
\begin{equation*}
(B+C)y=\|B+C\|y.
\end{equation*}
That is $B+C$ attains norm at $y$.
\end{proposition}
\begin{proof}
Let $\epsilon>0$. Since $B\geq 0$, there exists an anti-self-adjoint, unitary operator $J$ such that $JB=BJ$. Let $B_{+}$ be the unique operator on $H^{Jm}_{+}$ such that $\tilde B_{+}=B$. Now, by the classical theorem (\cite[Lemma 1]{enfloetal}), there exists a rank one positive operator, denote it by $C_{+}$ such that $\|C_{+}\|\leq \epsilon$ and $B_{+}+C_{+}\in \mathcal N(H^{Jm}_{+})$.  In fact, there exists $y\in N(B_+)^{\bot}$ such that $(B_{+}+C_{+})y=\|B_{+}+C_{+}\|y$. Now, by  Remark \ref{sumextension}, $B+C=\widetilde{B_{+}+C_{+}}$ and $B+C\in \mathcal N(H)$ by Lemma \ref{normattainingequivalnce}. It is clear that $C$ is a positive, rank one operator on $H$.

Also, we have that $(B+C)y=(B_{+}+C_{+})y=\|B_{+}+C_{+}\|y=\|B+C\|y$. As $N(B)=N(B_{+})$, we can conclude that $y\in N(B)^{\bot}$.
\end{proof}

Next, we extend proposition \ref{positivenormattaining} to the general case. For this purpose we use the polar decomposition of a quaternionic operator. Here we recall the details.
\begin{theorem}\label{normdensegeneral}
The set  $\mathcal N(H)$ is dense in $\mathcal B(H)$ with respect to the operator norm of $\mathcal B(H)$.
\end{theorem}
\begin{proof}
Let $T=V|T|$ be the polar decomposition of $T$.  Given $n\in \mathbb N$, By Proposition \ref{positivenormattaining}, there exists a rank one, positive  operator $C_n$  with $\|C_n\|\leq \frac{1}{n}$ such that $|T|+C_n\in \mathcal N(H)$.  In fact, there exists a sequence $(x_n)\subset H$ such that $(|T|+C_n)x_n=\||T|+C_n\|x_n$ for each $n\in \mathbb N$. Define $K_n:=VC_n$ for each $n\in \mathbb N$. Note that $x_n\in N(T)^{\bot}=N(V)^{\bot}$. Since
$V|_{N(V)^{\bot}}$ is an isometry, we have that $\|(T+K_n)x_n\|=\|V(|T|+C_n)x_n\|=\|(|T|+C_n)x_n\|=\||T|+C_n\|$. But $\|T+K_n\|\leq \||T|+C_n\|=\|(T+K_n)x_n\|$. This shows that $T+K_n$ attains norm at $x_n$.
\end{proof}

\begin{remark}
We can also prove that $\mathcal N(H_1,H_2)$ is norm dense in $\mathcal B(H_1,H_2)$ following the similar steps as in Theorem \ref{normdensegeneral}. For this purpose one has to obtain the polar decomposition theorem for $T\in \mathcal B(H_1,H_2)$. Again this can be done by following the similar steps in \cite[Theorem 2.20]{Ghiloni}.
\end{remark}
\bibliographystyle{plain}

\end{document}